\documentclass[11pt]{amsart}

\usepackage{cancel}
\usepackage{latexsym, amssymb, amsmath}
\usepackage{soul,esint}
\usepackage{amsfonts, graphicx}
\usepackage{graphicx,color}
\usepackage{mathtools}
\usepackage{hyperref}
\usepackage{verbatim}
\usepackage{faktor}
\usepackage{xfrac}
\makeatletter
\renewcommand*{\eqref}[1]{%
  \hyperref[{#1}]{\textup{\tagform@{\ref*{#1}}}}%
}
\makeatother

\newcommand{\be}{\begin{equation}}
\newcommand{\ee}{\end{equation}}
\newcommand{\beq}{\begin{eqnarray}}
\newcommand{\eeq}{\end{eqnarray}}

\newcommand{\dist}{\mathrm{dist}}

\newtheorem{thm}{Theorem}[section]

\newtheorem{prop}{Proposition}[section]
\newtheorem{cor}{Corollary}[section]
\newtheorem{defn}{Definition}[section]
\theoremstyle{remark}
\newtheorem{rem}{Remark}[section]
\numberwithin{equation}{section}

\def\be{\begin{equation}}
\def\ee{\end{equation}}
\def\bee{\begin{equation*}}
\def\eee{\end{equation*}}

\def\lf{\left}
\def\ri{\right}

\def\p{\partial}

\def\tr{\operatorname{tr}}

\def\a{{\alpha}}
\def\b{{\beta}}

\def\R{\mathbb{R}}

\def\d{\frac {\operatorname d}{\operatorname {dt}}}

\def\S{\Sigma}

\def\Lp{\Delta}
\def\Na{\nabla}
\def\eps{\varepsilon}

\def\d{\delta}

\def\MS{\mathcal{S}}

\def\m{\mathfrak{m}}

\def\l{\lambda}
\def\mE{\mathcal{E}}

\def\mE{\mathcal{E}}

\def\HH{\mathbb{H}}
\def\h {\hat}
\def\l {\lambda}
\def\s {\sigma}

\def\mN{\mathcal{N}}
\def\gH{g_{\HH}}
\def\mS{\mathbb{S}}
\def\tg{\tilde{g}}
\def\nS{\mathbb{S}}
\def\MS{\mathcal{S}}

\begin{document} 
\title[AH PMT AF Shield with arbitrary ends]
{Positive mass theorem for initial data sets with arbitrary ends}

\author{Tin-Yau Tsang}
\address [Tin-Yau Tsang] {Department of Mathematics, University of British Columbia, Vancouver, BC, V6T 1Z2}
\email{tytsang@math.ubc.ca}

\begin{abstract}
We showed a positive energy theorem for asymptotically flat initial data sets with the concept of spectral PSC by \cite{HSY, BHHSZ, BW1} and the Jang equation \cite{SY2, E3,Jang}.  Then, we proved a quantitative shielding theorem concerning the causal property of the energy-momentum vector of an asymptotically hyperbolic manifold. As a result, we established the positive mass theorem for complete asymptotically hyperbolic manifolds satisfying the dominant energy condition. As corollaries, we also obtained corresponding results for manifolds with asymptotically locally hyperbolic ends with a certain symmetry.  
\end{abstract}

\maketitle

\section{Introduction}

For general relativity, it is a fascinating question to seek global invariants to  indicate energy for a spacelike slice. Arnowitt, Deser and Misner defined the ADM mass in \cite{ADM}. The positive mass theorem proved by Schoen and Yau \cite{SY1} showed that asymptotically flat manifolds with non-negative scalar curvature must have non-negative ADM mass. Furthermore, they showed that Euclidean space is characterised by vanishing ADM mass. Witten (\cite{W}, cf. \cite{PT}) gave a proof on spin manifolds. 

\

For another standard model arising from Einstein's theory with negative cosmological constant, the hyperbolic space, Min-Oo \cite{MO} under certain strong assumptions has shown a scalar curvature rigidity theorem, and this is furthered by Andersson and Dahl \cite{AD}. Wang found a geometric invariant vector for energy-momentum \cite{XW} under a strong asymptotics which we referred to as Wang's asymptotics. It is then generalised by Chruściel and Herzlich \cite{CH03}. Corresponding to the positive mass theorem aforementioned, they \cite{XW, CH03} showed that the energy momentum vector for a spin asymptotically hyperbolic manifold with scalar curvature bounded from below by $-n(n-1)$ must be future-pointing timelike or vanishes. Moreover, the hyperbolic space is the only such with the vanishing energy-momentum vector. 

\

There are then different approaches for manifolds not necessarily spin. While assuming the sign of the mass aspect function, this was shown by Andersson, Cai and Galloway \cite{ACG} by patching with the hyperbolic space. Chruściel, Galloway, Nguyen and Paetz \cite{CGNP18} ruled out the past-pointing timelike case by approximating the metric by those with constant mass aspect function and Chruściel and Delay \cite{CD19} eliminated the spacelike case by the Maskit gluing. The Jang equation \cite{SY2} approach is adopted by Sakovich \cite{S} and Lundberg \cite{L}. On the other hand, a more general rigidity result is established by Huang, Jang and Martin \cite{HJM}. 

\

Motivated by the Liouville Theorem {\cite{SY4}, it is conjectured (\cite{SY4,SY5}) that for a complete manifold with non-negative scalar curvature and an asymptotically flat end $\mE$, the ADM mass of $\mE$ is non-negative. It is first proved by Lesourd, Unger and Yau in \cite{LUY21} with the assumption that $\mE$ is asymptotically Schwarzschild. This is generalised by Zhu {\cite{Zhu}}, Chen, Liu, Shi and Zhu {\cite{CLSZ} and Lee, Lesourd and Unger\cite{LLU22}. While for spin manifolds, this was done by Bartnik and Chruściel \cite{BC} and Cecchini and Ziedler \cite{CZ}. 

\

A natural question is whether the causal property of the energy momentum vector behaves similarly of an asymptotically hyperbolic end of a complete manifold with scalar curvature bounded from below by $-n(n-1)$. Chai and Wan \cite{CW} gave an affirmative answer for spin manifolds. 

\

In this paper, we are going to establish without the spin assumption the positive mass theorem for asymptotically hyperbolic manifolds with arbitrary ends. We first proved a quantitative shielding theorem by considering the Maskit gluing (\cite{CD19}), patching with the hyperbolic space (\cite{ACG}) and the positive energy theorem for asymptotically flat initial data sets (Theorem \ref{PET1},  c.f. \cite{SY2, E3},  the spacetime positive mass theorem with boundary \cite{GL, LLU21}) in particular its rigidity statement (\cite{HSY}).   

\begin{thm}[Quantitative shielding theorem, cf. \cite{LLU22} Theorem 1.1]\label{QST} 
Let $(M^n,g)$, $n \ge 4$, be an asymptotically hyperbolic manifold of H\"older type $(\a,\d)$ where $\a\in(0,1)$ and $\delta>\frac{n}{2}$. Let $U_0$, $U_1$, and $U_2$ be neighborhoods of an asymptotically hyperbolic end $\mathcal E$ such that  $\overline{U_2}\subset U_1$, $\overline{U_1}\subset U_0$, and $\overline {U_0 \setminus \mathcal E}$ is compact,
and let 
$$D_0=\dist_g(\partial U_0, U_1)\quad\text{and}\quad
D_1 =\dist_g(U_2,\partial U_1).$$
If the following hold:
\begin{enumerate}
    \item $g$ has no points of incompleteness in $U_0$, 
    
    \item $R_g\ge -n(n-1)$ on $U_0$, and 
    
    \item \label{scalar_bound} the scalar curvature satisfies the largeness assumption 
    \begin{equation}\label{largeness}
R_g+n(n-1)>\frac{4}{D_0 D_1}\quad \text{on}\quad\overline{U_1}\setminus U_2,
    \end{equation}
\end{enumerate}
then the energy momentum vector of $\mathcal E$ is future-directed causal or vanishing. 
\end{thm} 

There have been recent advancements of Riemannian positive mass theorem with arbitrary ends dealing with singularities by blowing up with suitable functions (He-Shi-Yu \cite{HSY}, Bi-Hao-He-Shi-Zhu \cite{BHHSZ} and Brendle-Wang \cite{BW1}) to carry out the dimension induction coupled with torical symmetrisation (\cite{FCS80}) in Schoen-Yau \cite{SY7}. By identifying the Jang graph (\cite{SY2, E3}, cf.\cite{Jang}) as a complete asymptotically flat manifold with positive scalar curvature in spectral sense, one can obtain the following. 

\begin{thm}[Quantitative shielding theorem, cf. \cite{LLU22} Theorem 1.1]\label{QST2} 
Let $(M^n,g, k)$, $n \ge 3$, be an asymptotically flat initial data set of Sobolev type.  Let $U_0$, $U_1$, and $U_2$ be neighborhoods of an asymptotically flat end $\mathcal E$ such that  $\overline{U_2}\subset U_1$, $\overline{U_1}\subset U_0$, and $\overline {U_0 \setminus \mathcal E}$ is compact,
and let 
$$D_0=\dist_g(\partial U_0, U_1)\quad\text{and}\quad
D_1 =\dist_g(U_2,\partial U_1).$$
If the following hold:
\begin{enumerate}
    \item $g$ has no points of incompleteness in $U_0$, 
    
    \item $tr_g k \leq 0$ on $U_0\setminus \overline {U_2}$    
   
    \item $\mu\geq |J|_g$ on $U_0$, and 
    
    \item \label{energyscalar_bound} the energy satisfies the largeness assumption 
    \begin{equation}\label{energylargeness}
\mu-|J_g|>\frac{4}{D_0 D_1}\quad \text{on}\quad\overline{U_1}\setminus U_2,
    \end{equation} 
\end{enumerate}
then the ADM energy $E$ is non-negative.  
\end{thm}

As a corollary,  with the density theorems in \cite{EHLS}, we obtain the following positive energy theorems in higher dimensions.   
\begin{thm} \label{PET1} 
Let $(M^n,g, k)$, $n \ge 3$, be an asymptotically flat initial data set of Sobolev type satisfying the dominant energy condition with only one end $\mathcal E$. 
The ADM energy $E$ is non-negative.  
\end{thm}

See Section 4 for detail and variant statements. 

\begin{rem}
According to the last remark following Section 6.3 in \cite{EHLS}, the positive mass theorem $E\geq|P|$ can be obtained by the positive energy theorems above. 
\end{rem}

\begin{rem}
During the preparation of this manuscript, Brendle and Wang \cite{BW2} showed a similar result to Theorem \ref{PET1} by analysing the regularised Jang equation. 
\end{rem} 

\begin{rem}
After this preprint was completed, we became aware of an independent preprint \cite{HKLZ} which proves closely related results for asymptotically flat and asymptotically hyperbolic initial data sets with one end.  
\end{rem}

For manifolds with boundary, we could show a shielding theorem concerning the mean curvature. 
\begin{thm}[cf. \cite{LLU22} Theorem 1.7]\label{CZ2}
Let $(M^n,g)$, $n\ge 4$, be a complete asymptotically hyperbolic manifold with nonempty compact boundary $\partial M$. Set $U_0=M$, and let $U_1$, $U_2$, $D_0$, $D_1$ be as in Theorem \ref{QST}. 
If
\begin{enumerate}
    \item $R_g\ge -n(n-1)$ on $M$, 
    \item \label{CZ2a} $R_g>-n(n-1)+ \kappa$ on $\overline U_1\setminus U_2$, where $\kappa \in (0, \frac{4}{D_0 D_1})$ (see Remark \ref{rk:a} below), and
    \item \label{CZ2c} the mean curvature of $\partial M$ with respect to the normal pointing into $M$ satisfies
\begin{equation}\label{Hineq}
H<n-1+\frac{2\kappa D_1}{4-\kappa D_0 D_1}, \end{equation}
\end{enumerate} 
then the energy momentum vector of $\mathcal E$ is future-directed causal or vanishing.
\end{thm}

\begin{rem}\label{rk:a}
Fixing $D_0$ and $D_1$ in \eqref{Hineq} and letting $\kappa\to 0$ recover the classical condition $H\le n-1 $. If  $ \kappa \ge \frac{4}{D_0 D_1}$, then Theorem \ref{QST} applies without taking the boundary mean curvature into account. 
\end{rem}

The following density theorem tells that an asymptotically hyperbolic end can be approximated by Wang's asymptotics which allow perturbations aforementioned.  

\begin{thm}[Density theorem, cf. \cite{LLU22} Theorem 1.3, \cite{DGS} Proposition B.1]
    \label{WANG}
    Let $(M^n,g)$, $n\geq 3$, be a Riemannian manifold with an asymptotically hyperbolic end $\mE$ of H\"older type $(\a,\d)$ where $\a\in(0,1)$ and $\delta>\frac{n}{2}$. Suppose that $R_g \geq -n(n-1)$. Then, for any $\varepsilon > 0$, $\d' \in (\frac{n}{2}, \d)$ and any compact set $K\subset M$, there exists a metric $\tg$ with the following properties:
    \begin{enumerate}
        \item $\tg$ has Wang's asymptotics. 

        \item $R_{\tg} \geq -n(n-1)$ in $M$ and $R_{\tg} \equiv -n(n-1)$ outside a bounded set in $\mathcal E$. 

        \item $\|g-\tg\|_{C^{2,\a}_{\d'}(K\cup\mE)} < \varepsilon$.  
        \item $\left|\tilde \m - \m \right| < \varepsilon. $
        \item The metrics $g$ and $\tilde g$ are $\eps$-close as bilinear forms everywhere on $M$. 
    \end{enumerate}
\end{thm}

Then, by completeness of manifolds, along with the construction of initial data sets in the proof of Theorem \ref{QST}, we obtained the following result. 

\begin{thm}\label{main}
Let $(M^n,g)$, $n\ge 4$, be a complete manifold with $R_g\geq -n(n-1)$, and suppose that it has at least one asymptotically hyperbolic end $\mathcal E$ of H\"older type $(\a,\d)$, where $\a\in(0,1)$ and $\delta>\frac{n}{2}$. Then the energy momentum vector of $\mathcal E$ is future-directed causal or vanishes. 
\end{thm}

As a consequence, we also proved an inextendibility result. 
\begin{cor}\label{CZ}
Let $(M^n,g)$,  $n\ge 4$, be a Riemannian manifold with an asymptotically hyperbolic end $\mathcal E$ of H\"older type $(\a,\d)$, where $\a\in(0,1)$ and $\delta>\frac{n}{2}$. 
 If the energy momentum vector of $\mathcal E$ is neither future-directed causal nor vanishes, then there exists a constant $D$, depending only on $\m(\mE,g)$ and $\|g-\gH\|_{C^{2,\a}_{\d}(\mathcal E)}$, with the following property. In the $D$-neighborhood $N_D(\mathcal E)$ of $\mathcal E$, at least one of the following holds:
\begin{enumerate}
 \item $R_g<-n(n-1)$ somewhere in $N_D(\mathcal E)$, or

 \item $N_D(\mathcal E)$ contains an incomplete point. 
\end{enumerate}
\end{cor}

\

This text is organised as follows. In Section \ref{Preliminaries}, some definitions of asymptotically hyperbolic manifolds and asymptotically flat initial data sets are reviewed.  In Section \ref{PIDS}, we construct an asymptotically flat initial data set form an asymptotically hyperbolic manifold. In Section \ref{AF},  Theorem \ref{QST2} and \ref{PET1} are shown. In Section \ref{shield}, Theorem \ref{QST} and Theorem \ref{CZ2} are proved. In Section \ref{complete}, completeness and inextendibility of asymptotically hyperbolic manifolds are discussed. In Section \ref{ALH}, corresponding results for asymptotically locally hyperbolic ends are shown.  In Appendix \ref{Density}, we prove Theorem \ref{WANG}.  %In Appendix \ref{Singular}, we discuss how to adopt Section 3 in \cite{BW1} for unbounded singular set.   

\

\textbf{Acknowledgements }: 
The author would like to thank Prof. Richard Schoen for his guidance, suggestions on Jang graph and his support in NSF grant DMS-2005431. The author is grateful to Prof. Greg Galloway for explaining \cite{ACG} and valuable comments on the manuscript.  The author would like to thank Shihang He and Kai-Wei Zhao for patiently answering his questions. The author is thankful to Prof. Erwann Delay for several suggestions which improve this article, Prof. Piotr Chruściel for answering questions regarding \cite{CGNP18} and Prof. Lan-Hsuan Huang for stimulating discussion.  The author is grateful to Banff International Research Station for their hospitality where part of this work is carried out. The author also thanks the hospitality of Harvard’s Black Hole Initiative and Martin Lesourd for bringing up and early contribution to this project. The author thanks Andoni Royo Abrego for suggesting the references \cite{G, DGS}.

\section{Preliminaries}\label{Preliminaries}
\begin{defn}\label{incomplete}
	Let $(X,d)$ be a metric space and $(\overline X,d)$ be its completion. For example, $\overline X$ can be constructed by taking appropriate equivalence classes of Cauchy sequences. A point in $\overline X\setminus X$ is called a \emph{point of incompleteness} for $X$. A set $S\subset X$ is said to be \emph{complete} if its closure in $X$ remains closed under the inclusion $X\to\overline X$.   
\end{defn}

\begin{defn}\label{strinfty}
Let $M^n$ be a noncompact manifold with a distinguished end $\mathcal E$. We say that $(M,g)$ possesses a \emph{structure of infinity} along $\mathcal E$ if $\mathcal E$ possesses no points of incompleteness and there exists a diffeomorphism 
\begin{equation}\Phi:\mathcal E\to \R^n\setminus B_{r_0}\end{equation} for some $r_0>0$
and the coordinate norm $r:=|x|$ diverges as we go out along the end. The set $\mathring M=M\setminus \mathcal E$ is called the \emph{core}. Note that in our definition, the core is not assumed to be compact and $(M,g)$ is not assumed to be complete. We will often identify $\mathcal E$ with the set $\{|x| \ge r_0\}$. Under $\Phi$, we can express
\begin{equation*}
    \gH = dr^2 + \sinh^2(r) g_{\mS^{n-1}}
\end{equation*}
the hyperbolic metric on $\R^n \setminus B_{r_0} \cong (r_0,\infty) \times \mS^{n-1}$,  which we extend arbitrarily to a complete metric on all of $M$ and denote by $\overline g$. We interchangeably denote $(\R^n\setminus B_{r_0}, \gH)$ by $(\HH^n\setminus B_{r_0}, \gH)$ when the context is clear.

We also allow for $M$ to have a boundary, but of course require $\partial M$ to not intersect~$\mathcal E$.
\end{defn}

\begin{defn}\label{weighted holder space} Let $(M^n, g)$, $\mathcal E$, and $\Phi$ be as in the previous definition.  Let $N$ be a closed subset of $M$ which contains $\mathcal E$ and such that $N\setminus \mathcal E$ is compact.
For $u:N\to \R$, we say that 
$u \in C_\delta^{k,\alpha}(N)$ for a non-negative integer $k$, and $0< \alpha \leq1$ if
\[
\left\| u \right\|_{C^{k, \alpha}_\delta(N)} 
:= \|u\|_{C^{k,\alpha}(N\setminus \mathcal E)}+\sup_{x \in \HH^n\setminus B_{r_0}} e^{\d r(x)} 
\left\| \Phi_* u \right\|_{C^{k, \alpha}(B_1(x))} < \infty.
\]
Here is an equivalent definition of the
$C^{k, \alpha}_\delta(\mE)$ norm. 
\[
\left\|u \right\|_{C^{k, \alpha}_\delta(\mE)} 
= \sum_{0\leq j \leq k} \sup_{\HH^n \setminus B_{r_0}} |e^{\delta r} \nabla^j (\Phi_* u)| 
+ \|e^{\delta r} \nabla^k(\Phi_* u)\|_{C^{0,\alpha}(\HH^n \setminus B_{r_0})}.
\] 
One can naturally extend the definition to tensor bundles. 
\end{defn}

\begin{defn} \label{defAHmanifold} 
Let $(M^n,g)$ be a noncompact smooth Riemannian manifold possessing a structure of infinity $\Phi$ along $\mathcal E$. Let $\a\in (0,1]$ and $\d>\frac{n}{2}$. 
We say that $(M, g)$ is an asymptotically hyperbolic 
manifold of H\"older type $(\a,\d)$, 
if there exists a diffeomorphism
\[
\Phi: \mE \to \HH^n \setminus B_{r_0}, 
\]
such that $h:=\Phi_* g - \gH \in C^{2, \a}_{\d}(\mE)$ and 
\begin{equation} \label{decay2}
\int_{\HH^n \setminus B_{r_0}} |R_g + n(n-1)| \cosh r \, d\mu_{\gH} < \infty. 
\end{equation}
\end{defn}

Furthermore, $g$ has \textit{Wang's asymptotics} if 
\begin{equation*}
    \Phi_*g = dr^2 + \sinh^2(r) g_r
\end{equation*}
where
\begin{equation*}
    g_r = g_{\nS^{n-1}} + m e^{-nr} + O^{k,\alpha}\bigl(e^{-(n+1)r}\bigr) \,,
\end{equation*}
is an $r$-dependent family of symmetric 2-tensors on $\mathbb{S}^{n-1}$, $ m\in C^{k,\a}$ is a symmetric 2-tensor on $\mathbb{S}^{n-1}$. Note that $\tr_{\mathbb{S}^{n-1}} m$ is called the mass aspect function. 

\

Let 
$\mN := \{ V \in C^{\infty}(\HH^n) \mid \Na^{\HH}\Na^{\HH} V = V g_\HH \}$.
This is a vector space of the static potentials of the hyperbolic space with a basis consisting of the functions 
\[
V_{(0)} = \cosh r, \, 
V_{(1)} = x^1 \sinh r, \dots , \,
V_{(n)} = x^n \sinh r,
\]
where the functions $x^1, \dots, x^n$ are the coordinate functions
on $\R^{n}$ restricted to $\mS^{n-1}$. 

The linear functional $H^g_{\Phi}$ on $\mN$ defined by (\cite{XW, CH03})
\[
H_{\Phi}^g (V)
= \lim_{r \to \infty} \int_{\mS_r} \left(
V (div_{\gH} h- d tr_{\gH} h) + (tr_{\gH} h) dV - h(\nabla^{\HH} V, \cdot)
\right) (\nu) \, d \mu_{\gH}
\] 
is called the {\em mass functional} of $(M,g)$ with respect to $\Phi$, where $\nu$ is the unit normal vector pointing to the infinity of $\mE$ with respect to $g_{\HH}$. The energy momentum vector $\m(\mE, g)=(m_0,m_1,...,m_n)\in \R^{n+1}$ is defined by $m_{\b}:=H^g_{\Phi}(V_{(\b)})$ for $\b=0,1,...,n$. 

\

Correspondingly, one can define asymptotically locally hyperbolic ends with different backgrounds and a scalar mass with different static potentials, see \cite{HJ} Section 2 for details. 

\

A 3-tuple $(M,g,k)$, where $g$ is a Riemannian metric and $k$ is a symmetric $(0,2)$-tensor, is called an initial data set. Define the conjugate momentum tensor by $\pi=k-(tr_g k)g$. Under constraint equations, we can define the mass density $\mu$ and the current density $J$ by
$$\mu=\frac{1}{2}(R_g+(tr_g k)^2-|k|_g^2)$$
and $$J=div_g(k-(tr_gk)g)=div_g\pi.$$
$(M,g,k)$ is said to satisfy the dominant energy condition if 
$$\mu\geq|J|_g.$$ 

An asymptotically hyperbolic manifold $(M,g)$ can be regarded as an initial data set $(M,g,-g)$. In this case, the dominant energy condition is translated into $$R_g\geq -n(n-1).$$ 

On the other hand, for asymptotically flat initial data sets (\cite{EHLS} Section 1, \cite{LLU21} Section 2), one can define the ADM energy-momentum vector $(E,P)$  
$$
E = \frac{1}{2(n-1)|\mathbb{S}^{n-1}|} \lim_{r \to \infty} \int_{|x| = r}  ( g_{ij,i} - g _{ii, j} ) \nu^j,
$$ 
$$P_i:=\frac{1}{(n-1)|\mathbb{S}^{n-1}|} \lim_{r \to \infty} \int_{|x| = r}  \pi_{ij} \nu^j, \,\ \,\ \,\ i=1,2,...,n,$$
where the outward unit normal $\nu$ and surface integral are with respect to the Euclidean metric. 

For a smooth closed hypersurface $S\subset M$, we say $S$ is a weakly outer trapped surface if on $S$, the outer null expansion 
$$\theta_+=H+tr_S k \leq 0,$$
and a marginally outer trapped surface ($MOTS$) if 
$$\theta_+=0;$$
correspondingly, 
a weakly inner trapped surface if the inner null expansion 
$$\theta_-=H-tr_S k \leq 0,$$
and a marginally inner trapped surface ($MITS$) if 
$$\theta_-=0,$$ 
where $H$ is computed with respect to the normal pointing to the infinity of the designated end $\mathcal{E}$. A surface is weakly trapped if it is either weakly outer trapped or weakly inner trapped.  

\begin{rem} In this note, the mean curvature is computed in the convention that $\mathbb{S}^2\subset \R^3$ has positive mean curvature with respect to the outward normal.  
\end{rem}

The positive mass theorem (\cite{EHLS}) states that for an asymptotically flat initial data set satisfying the dominant energy condition, $E\geq |P|$. While for initial data sets with weakly trapped boundary, it is studied in \cite{GL} and \cite{LLU21}. Moreover, if $E=0$, then $(M,g,k)$ is topologically $\R^n$ and can be isometrically embedded into Minkowski space (\cite{SY2, E3, LLU21}).  These results were stated up to dimension 7. And we will show corresponding statements in higher dimensions in Section \ref{AF}.  

\section{Preparation of initial data sets with negative mass}\label{PIDS} In this section, we are going to construct an initial data set which has a Euclidean end as motivated by \cite{CD19}. 

\

Let $n\geq 4$. Let $(M,g)$ be an asymptotically hyperbolic manifold with $R_g\geq -n(n-1)$. By Theorem \ref{WANG}, we can assume that it is with Wang's asymptotics. Assume on the contrary, its energy momentum vector $\m=(m_0, m_1,...,m_n)$ is neither causal future-pointing nor vanishing. Here we follow the strategy of \cite{CD19} and \cite{CGNP18}. 

\subsection{Construction of a metric with past-pointing timelike energy momentum vector}
If the energy momentum vector is spacelike, we can choose a coordinate such that 
\be
\begin{split}
\m=(m_0, m_1,0,...,0), 
\end{split}
\ee 
where $m_1 \geq |m_0|$ by the assumption of its causal property and $m_0<0$. 

\

Then, observe that Theorem 3.3 in \cite{CD15} is actually ``local", i.e. concerning only the conformal boundary. Thus for all $\eps>0$, it can be applied to $(M,g)$ to construct a metric $g_{\eps}$ on $M$ such that $g_{\eps}|_{U_{\eps}}=g_{\HH}$ and $g_{\eps}|_{U_{2\eps}}=g$, with $R_{g_{\eps}}\geq -n(n-1)$ on $M$. Here, $U_{a}$ denotes a coordinate half ball in the upper half space model centred on the conformal boundary with radius $a$ with respect to the compactified metric. Moreover, for $\b=0,1,...,n$, $m_{\eps,\b}-m_{\b}=o(\eps^{\frac{n-4}{2}})$ when $n\geq 5$ while $o(\eps^2)$ when $n=4$. (See \cite{CD19} Section 2.) 

\

Then, as described in \cite{CD19} Section 2, we consider 2 copies of $(M,g_{\eps})$ above, one can further impose a conformal transformation to obtain a half hyperbolic space on each of them. Then, we can glue the non-trivial parts along the boundary of the half hyperbolic space. We denote the resulted manifold by $(\hat M, \hat g)$ which is asymptotically hyperbolic with $R_{\hat g}\geq -n(n-1)$. 

\

One can follow the computation in \cite{CD19} Section 3 P.10 (3.8) to carry out another conformal transformation. We therefore obtain a timelike past-pointing energy-momentum vector denoted by $\hat \m$ if we choose $\gamma=\gamma(\m_{\eps})>0$ and $\eps$ aforementioned, both sufficiently small, where 

\begin{equation}
\begin{split}
\hat \m=\frac{2}{\sqrt{1-\cos ^2 \gamma}} \left( m_{\eps, 0}-\cos \gamma \, m_{\eps, 1}, \vec{0} \right). 
\end{split}
\end{equation}

Note that the Wang's asymptotics, if disturbed, can always be restored using Theorem \ref{WANG}.  

\subsection{Perturbation to a metric with an exactly hyperbolic end} By the previous subsection, it is sufficient to deal with the case where the energy-momentum vector is past-pointing timelike. Note that if $\hat{\m}$ is past-pointing timelike, one can apply a coordinate change as mentioned in \cite{CGNP18} so that 
\be 
\hat \m=(-\sqrt{\hat m_0^2-\sum_{i=1}^n \hat m_{i}^2},0,...,0).
\ee
 Then, one can apply \cite{CGNP18} Corollary 1.4 and choice of coordinates in \cite{CGNP18} Section 3 which only considers a tubular neighbourhood of the conformal boundary to perturb $\hat g$ to $\tilde g$ such that $R_{\tilde g}\geq-n(n-1)$ and $\tilde \m$ is close to $\hat \m$ with the mass aspect function being constant. In particular, $\tilde \m$ is still timelike past-pointing and thus $\tilde m_0$ is negative. Hence, the mass aspect function is pointwise negative. Then we can, for some $\rho_0>>1$ to be determined, further perturb the metric to $\tilde{g}^{\rho_0}$ by \cite{ACG} such that $\tilde g^{\rho_0}=\tilde g$ away from the conformal boundary and exactly hyperbolic near the infinity. To be precise, we here state the theorem for our application. 

\begin{thm}[\cite{ACG} Theorem 3.2] \label{ACGdeformation}
Let $(M,g)$ be asymptotically hyperbolic of the Wang's asymptotics with $R_g\geq -n(n-1)$. If the mass aspect function $tr_{\mS^{n-1}} m$ is a negative constant, then for any sufficiently large $\rho=\rho(tr_{\mS^{n-1}} m)$, there exists $g^{\rho}$ on $M$ with constants $a=a(\rho)\in (0,1)$ such that 
\begin{equation}
    g^{\rho} =
    \begin{cases}
      g, & \text{if}\ r\leq \rho \\
      g^{\rho,a}:=\frac{1}{1+\frac{r^2}{a}}dr^2+r^2h_0, & \text{if}\ r\geq 9 \rho. 
    \end{cases}
\end{equation}
Moreover, $R_{g^\rho}\geq -\frac{n(n-1)}{a}$ on $M$ and $a(\rho)\nearrow 1$ as $\rho\to \infty$.   
\end{thm}

\subsection{Patching to a Euclidean end}
Let $\mN:=\{ \rho \geq 9 \rho_0 \}\subset \mE$ of $(\hat M, \tilde g^{\rho_0})$ which now is hyperbolic. We can isometrically embed $\mN$ as a hyperboloid $\left(t= \sqrt{ a + |x|^2}\right)$ into Minkowski space as a part of the initial data set $(\hat M,\tilde g^{\rho_0},-\frac{\tilde g^{\rho_0}}{\sqrt{a}})$.  Thus, as pointed out by \cite{CD19},  we can extend from $\{ \rho= 10 \rho_0 \}$ onward in a spacelike manner to any asymptotically flat slice, in particular, we pick a constant time slice.

Denote this initial data set by $(\Check M, \Check g, \Check k)$.  On $\{ \rho \geq 9 \rho_0 \}$, $(\Check M, \Check g, \Check k)$ is in Minkowski space. And away from its flat infinity, i.e. $\{\rho< 9 \rho_0\}$, $(\Check g, \Check k):=(\tilde g,-\frac{\tilde g}{\sqrt{a}})$. Note that by Theorem \ref{ACGdeformation}, $R_{\tilde g^{\rho_0}}\geq -n(n-1)>-\frac{n(n-1)}{a}$ on $\{\rho< 9 \rho_0 \}$. Thus, $(\Check M, \Check g, \Check k)$ is an asymptotically flat initial data set satisfying the dominant energy condition with $E(\Check g)=|\Check P|=0$.  

\section{Positive energy theorem for asymptotically flat initial data sets}\label{AF}
In this section, we are going to see how the assumptions in Theorem \ref{QST2} allows us to reduce the situation to an initial data set with only one asymptotically flat end with a strictly outer trapped boundary.  

\

Let $0\leq h\in C^{\infty}(M)$ to be determined, we are going to construct an initial data set $(M, \h g:= g, \h k:=k-\frac{h}{n-1}g)$. We also want to seek a subset $M_0$ containing the Euclidean end and $\overline{U_1\setminus U_2}$ on which $(\h g, \h k)$ satisfies DEC. Thus, we first have to know how the local energy density and current density change accordingly.

\begin{prop}\label{newDEC}
Let $(M,g,k)$ be an initial data set. For a smooth function $h:M\to \R$, define another initial data set $(M,g,\hat k:=k-\frac{h}{n-1}g)$. We have, 
\be
\begin{split}
\h \mu = \mu+\frac{n}{2(n-1)}h^2-h\, tr_g k \,\,\, \text{and} \,\,\, \h J= J+dh. 
\end{split}
\ee 
\end{prop}

\begin{proof}
Let $f\in C^\infty(M)$, consider an initial data set $(M,g,\hat k:=k+fg)$. 
By definition, 
\begin{equation}
\begin{split}
\hat \mu :=& \frac{1}{2}\lf( R_g-|\h k|_g^2+(tr_g \hat k)^2 \ri)\\
%=& \frac{1}{2}\lf( R_g-\la k+fg,k+fg\ra+(tr_g k + nf)^2 \ri)\\
%=& \frac{1}{2}\lf( R_g -|k|_g^2+(tr_gk)^2 -nf^2-2ftr_gk+n^2f^2+2nf tr_g k\ri)\\
=&\mu+\frac{1}{2}n(n-1)f^2+(n-1)ftr_gk.  
\end{split}
\end{equation}

\begin{equation}
\begin{split}
\hat \pi:=&\, \h k-(tr_g \h k)g=\pi+(1-n)fg.\\
%=&k-(tr_g k)g+(1-n)fg\\ 
\end{split}
\end{equation}

\begin{equation}
\begin{split}
\hat J:=& div_{g}(\hat \pi)= J + (1-n) df. \\
%=& div_{g}(\pi)+ (1-n) div_g(fg)\\
%=& J + (1-n) div_g(fg)\\
%=& J + (1-n) g(\Na f, \cdot)\\
\end{split}
\end{equation}

If $f:=\frac{-1}{n-1}h$, then 
\begin{equation}
\begin{split}
\hat \mu 
=&\mu+\frac{1}{2}\frac{n}{(n-1)}h^2-h\, tr_gk.  \\
\hat J= & J+dh. \\
\end{split} 
\end{equation}

\begin{comment}
Therefore, 
\begin{equation}
\begin{split}
\h \mu+\h J(\nu)=&\mu+J(\nu)+\frac{1}{2}\frac{n}{(n-1)}h^2-h\, tr_gk+\Na_{\nu}h. 
\end{split}
\end{equation} 
\end{comment}
\end{proof}

\begin{proof}[Proof of Theorem \ref{QST2} and Theorem \ref{PET1}]

From Proposition \ref{newDEC}, it suffices to consider the following inequality to construct an appropriate $h$ such that $h=0$ on $U_2$ and on $M_0\setminus U_2$, 
\begin{equation}\label{stability1}
\begin{split}
- \frac{4}{D_0D_1}<&\frac{n}{n-1}h^2-2h\,tr_{g} k -2|\Na h|\\ 
\end{split}
\end{equation}
where $M_0:=\overline{\{h<\infty\}}\subset U_0$. 
By \cite{LLU22} Section 6 (cf. \cite{LUY21} Section 3), for all sufficiently small $\gamma>0$, one can seek $M_0$ and construct $h\geq 0$ such that on $\overline{U_1}\setminus U_2$, 
\begin{equation}\label{stability2}
\begin{split}
\frac{n}{n-1}h^2 -2|\Na h|> - \frac{4}{D_0D_1}(1-\gamma), 
\end{split}
\end{equation} 
which implies \eqref{stability1} since $tr_g k\leq 0$ on $U_0\setminus \overline{U_2}$ by assumption.  Moreover, on $\{h<\infty\}\setminus U_1$, we get 
\begin{equation}\label{stability3}
\begin{split}
\frac{n}{n-1}h^2 -2|\Na h|\geq0.  
\end{split}
\end{equation} 

Then, by the largeness assumption on energy,  the asymptotically flat initial data set $(M_0,g,\hat k)$ satisfies the dominant energy condition.  While for any surface $S$ near $\p M_0$,  $\theta^+=\hat H+tr_{S} \hat k= H_g + tr_{S}k -h$. 
By the fact that $h$ is approaching infinity near $\p M_0$, any surface homologous and close to $\p M_0$ would have $\theta^+<0$.  Thus, without loss of generality,  one can take a subset of $M_0$, with an abuse of notation, still denoted by $(M_0, g, \hat k)$ as an initial data set with a smooth strictly outer trapped boundary.  Then by Theorem 3.7 in \cite{LLU21} , we can further assume $(M_0,g,\hat k)$ is of harmonic asymptotics and strict dominant energy condition.  Then, by the analysis in \cite{E10} Section 3 and \cite{E3} Section 2 (cf. \cite{SY2} Section 4), one can solve for the Jang equation with blow up whose solution is with a graph with an asymptotically flat end with the ADM energy same as $(M_0,g)$ and have a strictly positive scalar curvature in the strong spectral sense (see \cite{HSY} Definition 1.6, c.f.  \cite{BW1} Definition 1.2).  Thus, by \cite{BW1} Theorem 1.4, $E(g)\geq 0$.  While if there is only one asymptotically flat end, one can see $U_0=M$ and use Theorem 18 in \cite{EHLS} for perturbation. Result then follows as aforementioned.  
\end{proof} 

\begin{rem}
By \cite{E10} Lemma 2.1, local boundedness of the graph of the limit of the regularised Jang equation implies local uniform $C^1$ bound of the solution to the regualrised Jang equation and by elliptic regularity, we obtain a local uniform $C^3$ bound. Then by \cite{E10} Lemma 2.3 (cf. \cite{SY2} Proposition 2), the graphical component of the solution to the Jang equation is regular and complete.  
\end{rem}

The proof above also implies the following. 
\begin{cor} \label{PET2} 
Let $(M^n,g, k)$, $n \ge 3$, be an asymptotically flat initial data set of Sobolev type with boundary composed of either weakly outer or inner trapped surfaces satisfying the dominant energy condition with only one end $\mathcal E$.  
The ADM energy $E$ is non-negative.  
\end{cor} 

In turn, it implies the following. 
\begin{cor} \label{PET3} 
Let $(M^n,g, k)$, $n \ge 3$, be an asymptotically flat initial data set of Sobolev type with multiple ends, all are asymptotically flat.   
The ADM energy of each end is non-negative.  
\end{cor} 
\begin{proof}
Say $(M,g,k)$ has $l$ ends.  Fix one of them, say $N_1$, then large coordinate spheres on other ends $N_2,...,N_l$ would be smooth outer trapped surfaces with respect to $N_1$, then Corollary \ref{PET2} can be applied to $N_1$.  
\end{proof}

\section{Shielding in the interior and on the boundary}\label{shield}
\begin{proof}[Proof of Theorem \ref{QST}]Let $(M,g)$ be an asymptotically hyperbolic manifold satisfying the assumption in Theorem 1.1. Assume on the contrary that its energy momentum vector $\m$ is neither causal future pointing nor vanishing. As discussed in Section \ref{PIDS}, we can construct an asymptotically flat initial data set $(\Check M, \Check g, \Check k)$ satisfying the dominant energy condition as in the last step we pick $\rho_0$ sufficiently large such that $\overline{(U_0\setminus U_2)} \cap \{\rho>\frac{1}{2}\rho_0\}=\varnothing$. Note that along the construction, there are 2 copies of $U_0\setminus U_2$ in $\Check M$, yet we here notate their union also by $U_0\setminus U_2$.

By construction, On $U_1\setminus U_2$, we have 
\be 
2(\Check \mu- |\Check J|_{\Check g})= R_g+\frac{n(n-1)}{a}> R_g+n(n-1)>\frac{4}{D_0D_1}. 
\ee 
And on $(U_1\setminus U_2) \cap \{\rho<9\rho_0\} $,  we obviously have, 
\be 
2(\Check \mu- |\Check J|_{\Check g})= R_g+\frac{n(n-1)}{a}> 0.  
\ee 
Moreover, on $U_0\setminus \overline{U_2}$, $tr_{\Check g} \Check k <0$.  

Thus, one can use Section \ref{AF} to construct a function $h$ and an auxilary asymptotically flat initial data set with strictly outer trapped boundary $(M_0,\hat g:=\Check g,\hat k:=\Check k-\frac{h}{n-1}\Check g)$ on which we solve the Jang equation. Let $(M_{G}, g_{G})$ denote the Jang graph component with an asymptotically flat end with $E(g_{G})=E(\hat g)=0$.  By construction as aforementioned, with equation (11) in \cite{E3} (cf. \cite{SY2} P.256),  as in the notation \cite{HSY} Definition 1.6, we have that $(M_{G}, g_{G})$ is a complete manifold with $\frac{1}{2}$-scalar curvature no less than $\phi:=2(\hat\mu-|\hat J|_{\hat g})$ in the strong spectral sense with $\phi\geq 0$ and strictly positive somewhere.  The proof of Theorem 1.8 (1) (in particular Lemma 4.3) in \cite{HSY} utilised torical symmetrisation (\cite{FCS80, SY7}) to reduce the situation to a positive mass theorem with pointwise non-negative scalar curvature. By \cite{LLU22} Theorem 1.3 and \cite{BW1} Corollary 1.5, one can see that \cite{HSY} Theorem 1.8 also true for dimensions greater than or equal to 3.  In particular,  since $\phi>0$ somewhere, we know that $E(g_{G})>0$, contradiction arises. 

\begin{comment}
Note that first of all, a large coordinate sphere in the Euclidean end has $\theta^+=H_{g_{Euc}}>0$. While for any surface $S$ near $\p M_0$,  $\theta^+=\hat H+tr_{S} \hat k= H_g + tr_{S}k -h$. 
By the fact that $h$ is approaching infinity near $\p M_0$, any surface homologous and close to  $\p M_0$ would have $\theta^+<0$. 

Therefore, by \cite{E10, AM}, there exists a separating MOTS $\S$ in $M_0$, denote the non-compact component containing the Euclidean end by $M_1$. Therefore, $(M_1, \h g, \h k)$ is an asymptotically flat initial data set satisfying DEC with a Euclidean end and MOTS boundary $\S$. It is contradicting \cite{LLU21} Section 4.3, the rigidity of the spacetime positive mass theorem with boundary with zero ADM energy. The proof of Theorem \ref{QST} is complete.  
\end{comment}
\end{proof}

\begin{proof}[Proof of Theorem \ref{CZ2}]
Let $(M,g)$ be an asymptotically hyperbolic manifold satisfying the assumption in Theorem \ref{CZ2}. Assume on the contrary that its energy momentum vector $\m$ is neither causal future pointing nor vanishing. We now only have to fine tune the last step from the proof above. In particular after the modification of $(M,g)$ and regard it as an initial data set $(\Check M, \Check g, \Check k)$. Following \cite{LLU22} Section 6, one can construct $h$ with a more delicate choice of parameter to force $\p M$ to be a strictly outer trapper surface for the asymptotically flat initial data set $(\h M:=\Check M, \h g:= \Check g, \hat k:=\Check k-\frac{h}{n-1}g)$ which satisfies DEC and contains a Euclidean end. Note that in the proof of Theorem \ref{WANG}, the conformal factor $u$ can be chosen to be as close to 1 as desired. Hence, from the formula of the mean curvature after conformal change 
\begin{equation}
H_{u^{\frac{4}{n-2}}g}=u^{\frac{-2}{n-2}}\left(H_g+\frac{2(n-1)}{n-2}\frac{\p_{\nu} u}{u}\right) , 
\end{equation}
the strict upper bound of mean curvature can still be preserved after approximation. Again, it is contradicting \cite{LLU22} Section 4.3. The proof is thus complete. 
\end{proof}

\section{Completeness and inextendibility}\label{complete}
\begin{proof}[Proof of Theorem \ref{main}]
Let $(M,g)$ be a complete manifold with an asymptotically hyperbolic end $\mE$ satisfying $R_g\geq -n(n-1)$. We are going to prove Theorem \ref{main} by contradiction. Assume on the contrary that its energy momentum vector $\m$ is neither causal future pointing nor vanishing. 
We pick an arbitrarily large $\rho_0$ at the last step of the deformation of $(M,g)$ into $(\Check M, \Check g, \Check k)$ carried out in Section \ref{PIDS}. Note that along the deformation, $\Check g$ is still complete since it is only altering the infinity of the asymptotically hyperbolic end.

Now, on $W:=\{\rho\leq \rho_0\}$, we have 
\be 
\begin{split}
\Check \mu- |\Check J|= & R_g+\frac{n(n-1)}{a} \\
\geq & R_g+n(n-1)+\frac{n(n-1)}{a}-n(n-1)\\
\geq & n(n-1)(\frac{1-a}{a})=:\gamma>0. 
\end{split}
\ee 

If $W$ is compact, then we already have classical positive mass theorem for asymptotically hyperbolic manifolds. Hence, here we assume that it is non-compact. Now, one can set $U_1=\{\rho > \rho_0-1\}$ and $U_2=\{\rho > \rho_0\}$. Let $U_0$ be a set containing $\overline{U_1}$ large enough such that $D_0$ satisfies 
\be 
\begin{split}
\gamma>\frac{4}{D_0D_1}. 
\end{split}
\ee 
Note that it is legitimate since $M$ is complete. Theorem \ref{main} now follows from the proof of Theorem \ref{QST} in Section \ref{shield}. 
\end{proof}

Now, we are ready to show the inextendibility result. 

\begin{proof}[Proof of Corollary \ref{CZ}]
For a compactly supported function $f$, as discussed in Appendix \ref{Density}, we can solve the following equation for large $\l$ and sufficiently small $\tau>0$, 
\[-\frac{4(n-1)}{n-2}\Delta_{g_{\lambda}} u_\tau + R_{g_\l} u_{\tau}=-n(n-1) u_{\tau}^{\frac{n+2}{n-2}}+(\chi_{\lambda}(R_g+n(n-1))+\tau f).\]
Then the scalar curvature of $\tilde g_{\tau}=u_{\tau}^{\frac{4}{n-2}}g_{\l}$ satisfies 
\[R(\tilde g_\tau)+n(n-1)=(\chi_{\lambda}(R_g+n(n-1))+\tau f)u^{-\frac{n+2}{n-2}}_\tau.\] 

Pick $\l$ sufficiently large and $\tau$ sufficiently small that $\m(\tilde g_{\tau},\mE)$ is still spacelike or past-pointing. Then we know that the hypotheses of Theorem \ref{QST} are violated in $N_D^{\tilde g}(\mathcal E)$ for $\tilde g_{\tau}$, where $D$ is such that \eqref{largeness} is satisfied on some annular region where $f>0$. Moreover, for $\tau$ small, $u_\tau$ is close to 1. Note that $\tilde g_{\tau}$ does not generate any new points violating the dominant energy condition or incompleteness, so one of the hypotheses must first be violated for $g$ in $N^{\tilde g}_D(\mathcal E)$. Furthermore,  $g$ and $\tilde g$ are uniformly equivalent, $N_D^{\tilde g}(\mathcal E)\subset N^g_{D'}(\mathcal E)$ for some $D'$ close to $D$. 
\end{proof}

\section{Asymptotically locally hyperbolic ends}\label{ALH}
Now, we can consider a Riemannian manifold with an asymptotically locally hyperbolic end $\mE$ which has topology of $[0,1]\times \mathbb{S}^{n-1}/\Gamma$, where $\Gamma$ is a finite subgroup of $SO(n-1)$. Then by considering its universal cover as in \cite{CG21}, one can obtain the following result. 

\begin{cor}\label{QSTALH} 
Let $(M^n,g)$, $n\ge 4$, be an asymptotically locally hyperbolic manifold of H\"older type $(\a,\d)$ where $\a\in(0,1)$ and $\delta>\frac{n}{2}$. Let $U_0$, $U_1$, and $U_2$ be neighborhoods of an asymptotically locally hyperbolic end $\mathcal E$ with topology $[0,1]\times \mathbb{S}^{n-1}/\Gamma$ such that  $\overline{U_2}\subset U_1$, $\overline{U_1}\subset U_0$, and $\overline {U_0 \setminus \mathcal E}$ is compact,
and let 
$$D_0=\dist_g(\partial U_0, U_1)\quad\text{and}\quad
D_1 =\dist_g(U_2,\partial U_1).$$
If the following hold:
\begin{enumerate}
    \item $g$ has no points of incompleteness in $U_0$, 
    
    \item $R_g\ge -n(n-1)$ on $U_0$, and 
    
    \item \label{scalar_boundALH} the scalar curvature satisfies the largeness assumption 
    \begin{equation}\label{largenessALH}
R_g+n(n-1)>\frac{4}{D_0 D_1}\quad \text{on}\quad\overline{U_1}\setminus U_2,
    \end{equation}
\end{enumerate}
then the mass of $\mathcal E$ is non-negative. 
\end{cor}

\begin{thm}\label{mainALH}
Let $(M^n,g)$, $n \ge 4$, be a complete manifold with $R_g\geq -n(n-1)$, and suppose that it has at least one asymptotically locally hyperbolic end $\mathcal E$ of H\"older type $(\a,\d)$, where $\a\in(0,1)$ and $\delta>\frac{n}{2}$. Then the mass of $\mathcal E$ is non-negative. 
\end{thm}

Correspondingly, the adaptation also applies to Theorem \ref{CZ2} and Corollary \ref{CZ}.

\appendix
\section{Proof of Theorem \ref{WANG}}\label{Density}
We follow the set up of exhaustion as in \cite{LLU22} Section 2,  in particular Lemma 2.8, there exists a smooth proper function $\rho : M \to (0,\infty)$ such that
\begin{enumerate}
    \item $\rho(x) = r$ for any $x \in \mathcal E$ and $\rho < r_0$ in $M \setminus \mathcal E$.
    \item $\rho(p_i) \to 0$ for any sequence $\{p_i\} \subset M \setminus \mathcal E$ that eventually leaves every compact set.
\end{enumerate}
Let $\sigma > 0$ be a regular value of $\rho$, we define $M_\sigma \coloneqq \{\rho \geq \sigma\}$, a smooth manifold with boundary.  As in \cite{DGS} Appendix B, for $\lambda \geq r_0$, we can define a smooth cut-off function $0 \leq \chi_\lambda \leq 1$ equal to one in $\overline{M \setminus M_\lambda}$ and zero in $M_{\lambda+1}$, and define the metric
\begin{equation}
    g_\lambda \coloneqq \chi_\lambda g + (1-\chi_\lambda) \bar g
\end{equation}
on $M$. 

Let $\Delta_{\l}$ and $R_{\l}$ denote the Laplace-Beltrami operator and the scalar curvature with respect to $g_\lambda$. To preserve a non-zero lower bound of the scalar curvature after conformal change, instead of considering conformal Laplacian as in \cite{LLU22}, we seek a solution to the following equation as in \cite{DGS} Appendix B.  

\begin{equation}
    \label{DGS}
        -a_n \Delta_{\lambda} u_{\lambda} + R_\lambda u_{\lambda} = -n(n-1) u_{\lambda}^{\frac{n+2}{n-2}} + \chi_{\l}(R_g+n(n-1)) \quad
\end{equation}
where $a_n:= 4 \tfrac{n-1}{n-2}$. 
From this, for $\tilde g_\l:=u_{\l}^{\frac{4}{n-2}}g_{\lambda}$, we have
\begin{equation}
    R_{\tilde g_{\l}}+n(n-1)=u^{-\frac{n+2}{n-2}}\chi_\lambda (R_g+n(n-1)).  
\end{equation}

Let $\l>r_0$. As in \cite{LLU22}, we first consider the following Neumann problem on $M_{\s}$ for each $\sigma\leq r_0$ being a regular value of $\rho$, 
\begin{equation}
    \label{exhaustion}
    \begin{dcases}
        -a_n \Delta_\lambda u_{\lambda,\sigma} + R_\lambda u_{\lambda,\sigma} = -n(n-1) u_{\lambda,\sigma}^{\frac{n+2}{n-2}} +  \chi_{\l}(R_g+n(n-1))  &\text{in } \; M_\sigma 
        \\
        \partial_\nu u_{\l,\sigma} = 0  &\text{on } \; \partial M_\sigma. 
    \end{dcases}
\end{equation}

As illustrated by \cite{DGS} Appendix B, the existence of solution $u_{\l,\sigma}$ established by standard monotonicity method (\cite{G}) follows from the existence of constant subsolutions and supersolutions (barriers).  And this marks the difference between Theorem 1.3 in \cite{LLU22} and Theorem \ref{WANG} here since we have to impose control on scalar curvature to get subsolutions and supersolutions, precisely $R_g+n(n-1)\geq 0$.  In particular, as shown in \cite{DGS} Appendix B, the following constant barriers suffice. 
\begin{prop}
\label{bound}
For each sufficiently small $\tau > 0$, there exists some $\lambda_0 =\lambda_0(n, \tau, \delta, \|g - \bar g\|_{C^{2,\alpha}(\mathcal E)})> r_0$ such that for all $\lambda > \lambda_0$, the constants $u^{\pm} \coloneqq 1 \pm \tau$ are barriers to \eqref{exhaustion}. In particular, $\lambda_0 \to \infty $ as $\tau\to 0$.  
\end{prop}

A crucial observation is that $\lambda_0$ and $\tau$ are independent of $\sigma$ since the barriers only concerns the asymptotic behaviour on $\mE$.  Thus $\|u_{\l,\sigma}\|_{C^0(M_{\sigma})}\leq 1+\tau $,  independent of $\sigma$. To obtain a global solution, one wants to extract a subsequential limit and thus we have to establish a uniform bound. 

\begin{prop}\label{DNM}
If $\Omega'\subset\subset\Omega\subset M_{\sigma}$,  let $u_{\l,\sigma}$ be solution to \eqref{exhaustion}, then there exists $\beta=\beta(\Omega,\Omega')\in(0,\alpha]$ such that
\begin{equation}\label{Holder}
||u_{\l,\sigma}||_{C^{2,\b}(\Omega')}\leq C(\Omega,\Omega')\left( ||u_{\l,\sigma}||_{C^{0}(\Omega)} + ||\chi_{\l}(R_g+n(n-1))||_{C^{0,\b}(\Omega)} \right).  
\end{equation}
\end{prop}
\begin{proof} 
Let $\Omega' \subset\subset \hat \Omega \subset\subset \Omega$, then on $\Omega$, by \cite{GT} Theorem 8.24, there exists $\beta(\Omega,\Omega')$ such that, 
$$||u_{\l,\sigma}||_{C^{0,\b}(\hat \Omega)}\leq C(\hat \Omega, \Omega)\left( ||u_{\l,\sigma}||_{L^2(\Omega)}+||u_{\l,\sigma}^{\frac{n+2}{n-2}}||_{L^2(\Omega)}+||\chi_{\l}(R_g+n(n-1))||_{L^2(\Omega)}\right).$$ Without loss of generality, $\beta\leq \alpha$.

While on $\hat \Omega$, we can obtain by \cite{GT} Corollary 6.3, 
$$||u_{\l,\sigma}||_{C^{2,\b}(\Omega')}\leq C(\Omega',\hat \Omega)\left( ||u_{\l,\sigma}||_{C^{0}(\hat\Omega)} + ||u_{\l,\sigma}^{\frac{n+2}{n-2}}||_{C^{0,\b}(\hat \Omega)} + ||\chi_{\l}(R_g+n(n-1))||_{C^{0}(\hat\Omega)}  \right).$$
The result thus follows.  
\end{proof} 

By Proposition \ref{DNM} and Proposition \ref{bound}, we see that there exists a subsequence $\{u_{\l,\sigma_k}\}$ such that $u_{\l,\sigma_k}\to u_{\l}$ as $\s_k \to 0$, where $u_{\l}\in C^{2}_{loc}(M)$ satisfying \eqref{DGS}.  While on $\mE$,  note that $g_{\l}$ and $\bar g=g_{\HH}$ are uniformly equivalent independent of $\l$. Thus, by \eqref{Holder} there exists $\beta=\beta(\mE)>0$ such that 
$$||u_{\lambda}||_{C^{2,\b}(\mE)}$$ is bounded independent of $\l$.  If $\b < \a$, one can carry out interior Schauder estimate leveraging the regularity of $g$ to conclude that $||u_{\lambda}||_{C^{2,\a}(\mE)}$ is bounded independent of $\l$.

On the other hand, $1-\tau \leq u_{\lambda,\sigma_k} \leq 1+\tau$ for all $k\in \mathbb{N}$, we have that $u_\lambda$ is uniformly positive on $M$ and the conformal metric $\tilde g_{\lambda} \coloneqq u_\lambda^{\frac{4}{n-2}}g_\lambda$ is well-defined. 

\

For closeness and asymptotics, we also consider the equation satisfied by $v_{\l}:=u_{\l}-1$, 
\begin{equation}\label{closeness}
\begin{split}
-a_n\Lp_{\l}v_{\l}+nv_{\l}+(R_{\l}+n(n-1))v_{\l}=& \chi_{\l}(R_g+n(n-1))-(R_{\l}+n(n-1))\\
&-n(n-1)\left((1+v_{\l})^{\frac{n+2}{n-2}}-1-\frac{n+2}{n-2}v_{\l} \right). 
\end{split}
\end{equation}

Then as in \cite{DGS} Appendix B (cf. \cite{G} Section 3, \cite{Lee}),  $v_{\lambda}\in C^{2,\a}_{-n}(\mE)$.  Moreover,  for $\delta'\in (\frac{n}{2}, \delta)$, $||\tilde g_{\lambda}-g||_{C^{2,\a}_{\delta'}(\mE)}\to 0$ as $\lambda\to \infty$.  Let $K\subset M\setminus \mE$, what remains is to check smallness of $v_{\lambda}$ on $K$.  Note that the analysis is no different even if $K$ is ``far" from $\mE$. As pointed out after Proposition \ref{bound},  the $C^0$ bound is global on $M$ and we can pick $\lambda$ large such that $||v_{\l}||_{C^0(M)}$ is small and thus  $C^{2,\a}$ smallness follows as in the classical setting.   

Since, $v_{\lambda}\in C^{2,\a}_{-n}(\mE)$, then one can follow \cite[Theorem 5.2]{DS15} to apply a change of radial coordinate to get Wang's asymptotics while preserving the energy-momentum vector.  Finally, $|\m(\tilde g_{\lambda}) -\m(g)|$ can be made arbitrarily small as $\lambda \to \infty$ following from the closeness of metrics by standard computation.   

\section{Handling possibly unbounded singular set}\label{Singular}
Here we discuss some sufficient conditions to allow the design of \cite{BW1} of blow up functions to be applied to unbounded singular sets as of independent interest.  For the convenience of readers, the notations here are from \cite{BW1}.  Also see the treatment in \cite{HKLZ} Theorem 5.2. 

One of the keys of \cite{BW1} is the Minkowski dimension bound on the singular set by \cite{CN} for minimisers of area-related functionals.  However, if $\S$ is an almost minimising current (almost monotonicity and bounded mean curvature) with multiplicity one (see \cite{GS} for counterexamples), then one can still use \cite{NV} Theorem 1.6 to conclude the same Minkowski dimension (actually Minkowski content) bound on the singular set.  Moreover,  note that the coefficient depends on the area of the surface and ambient geometry.  

In particular,  if the singular set $\mathcal{S}$ is unbounded.  Then, one can decompose the portion containing the singular set into components which have a uniform area bound.  Let $A>0$,  $\S_\MS:=\bigcup_{i} \bar \S_i$, where $|\S_i|\leq A$ and $\bar \S_i \cap \bar \S_j=\emptyset$ if $|i-j|>1$, $\bar \S_i \cap \bar \S_{i+1}=\partial \S_{i}\cap\partial \S_{i+1}$.  

If $|\Na \log \hat \rho|$ and $|\Phi|$ are uniformly bounded and $Q$ is uniformly strictly away from 0 on $\S_{\MS}$ and a $\delta$-neighbourhood of $\S_{\MS}$ in $M$ has bounded geometry,  then we can conclude the following.  

For each $\MS_i:=\MS \cap \bar \S_i$, one can construct $\Psi_i$ and at the end, we would define $\Psi:= \sum_{i} \Psi_i$. Note the following for clarifying certain adjustment.  
\begin{enumerate} 
\item\label{u1} For Lemma 3.32 in \cite{BW1},  $C_0, C_1$ is independent of $i$.  Thus, so is $C$ in its statement and the choice of $t_0$ can be made independent of $i$. 
\item\label{u2} $C(q)$ on P.32 Theorem 3.33 in \cite{BW1} is independent of $i$ from the discussion above of uniform area bound on $|\S_i|$.  
\item\label{ul0} For \cite{BW1} Proposition 3.35,  since the support of $\Psi_i$ and $\Psi_j$ may have overlap if $|i-j|=1$.  (Without loss of generality $\sqrt{t_*}<<1$.)  But due to items $\ref{u1}$ and $\ref{u2}$ above.  We can still choose $l_0$ independent of $i$.  
\item\label{uc0c1} \cite{BW1} Proposition 3.34 and Corollary 3.36 is linear in nature.  Hence, they hold for $\Psi$ and moreover by item (\ref{ul0}) we know that ``sufficiently small" is uniform along $\S$.  Thus, the choice of $\eps_0$ is still valid as on \cite{BW1} P.35.  
\item By items (\ref{u1}), (\ref{u2}) and (\ref{uc0c1}), for \cite{BW1} Proposition 3.40, $c_0$ and $c_1$ is uniform along $\MS$ and hence it is still valid. 
\end{enumerate}

\begin{rem}
Here, the focus is that by \cite{NV}, the uniform area bound allows the constants to be uniformly chosen for the exact same construction of each function.  While in \cite{HKLZ} Section 5 Theorem 5.2 applies scaling to each function to allow the choice of a constant. 
\end{rem}

\textbf{Acknowledgements}: The author is grateful to Anna Skorobogatova, Paul Minter and Kai-Hsiang Wang for discussions on geometric measure theory.


\begin{thebibliography}{40} 
\bibitem{AD}
L. Andersson and M. Dahl, {\em Scalar curvature rigidity for asymptotically locally hyperbolic manifolds} Ann. Global Anal.Geom. \textbf{16} (1998), 1–27. 

%\bibitem{AMS05}
%L. Andersson, M. Mars, and W. Simon, {\em Local existence of dynamical and trapping horizons}, Physical review letters \textbf{95} (2005), no. 11, 111102.

\bibitem{ACG} L. Andersson, M. Cai, and G.J. Galloway, {\em Rigidity and positivity of mass
for asymptotically hyperbolic manifolds}, Ann. Henri Poincaré (2008) no. 9, 1–33. 

\bibitem{AM}
L. Andersson, and J. Metzger, {\em The area of horizons and the trapped region}, Communications in Mathematical Physics \textbf{290} (2009), no. 3, 941-972.

%\bibitem{AMS08}
%L. Andersson, M. Mars, and W. Simon, {\em Stability of marginally outer trapped surfaces and existence of marginally outer trapped tubes}, Adv. Theor. Math. Phys. \textbf{12} (2008), no.4, 853-888. 

\bibitem{ADM}
R. Arnowitt, S. Deser, and C. Misner, {\em Energy and the criteria for radiation in general relativity}, Physical Review \textbf{118} (1960), no. 4, 1100.

\bibitem{BC}
R. Bartnik and P.Chruściel, {\em Boundary value problems for dirac type
equations}, J. reine angew. Math. \textbf{579} (2005), 13–73. 

\bibitem{BHHSZ} Y. Bi, T. Hao, S. He, Y. Shi and J. Zhu. {\em A proof for the Riemannian positive mass theorem up to dimension 19}, arXiv preprint, arXiv:2603.02769 (2026). 

\bibitem{BW1} S. Brendle,  and Y. Wang.  {\em A dimension descent scheme for the positive mass theorem in arbitrary dimension}, arXiv preprint, arXiv:2604.08473 (2026). 

\bibitem{BW2} S. Brendle,  and Y. Wang.  {\em On the spacetime positive energy theorem in arbitrary dimension}, arXiv preprint, arXiv:2604.18561 (2026). 

\bibitem{CLZ}
S. Cecchini, M. Lesourd, and R. Zeidler, {\em Positive mass theorems for spin initial data sets with arbitrary ends and dominant energy shields}, International Mathematics Research Notices 2024, no. 9 (2024), 7870-7890.

\bibitem{CZ}
S. Cecchini, and R. Zeidler, {\em The positive mass theorem and distance estimates in the spin setting}, Transactions of the American Mathematical Society \textbf {377} (2024), no. 08, 5271-5288.

\bibitem{CW}
X. Chai, and X. Wan, {\em The mass of an asymptotically hyperbolic end and distance estimates}, Journal of Mathematical Physics \textbf{63} (2022), no. 12.

\bibitem{CN}
J. Cheeger and A. Naber, {\em Quantitative stratification and the regularity of harmonic maps and minimal currents}, Comm. Pure Appl. Math.  \textbf{66}, 965–990 (2013)

\bibitem{CLSZ}
J. Chen, P. Liu, Y. Shi, and J. Zhu, {\em Incompressible hypersurface, positive scalar curvature and positive mass theorem}, Mathematische Annalen (2025): 1-80. 

%\bibitem{CL20}
%O. Chodosh, and C. Li, {\em Generalized soap bubbles and the topology of manifolds with positive scalar curvature}, arXiv:2008.11888 (2020).

\bibitem{CD09} P.T. Chruściel and E. Delay, {\em Gluing constructions for asymptotically hyperbolic manifolds with constant scalar curvature}, Communications in Analysis and Geometry \textbf{17} (2009) no.2, 343-381.  

\bibitem{CD15} P.T. Chruściel and E. Delay, {\em Exotic hyperbolic gluings}, Journal of Differential Geometry \textbf{108} (2018) no.2, 243–293.  

\bibitem{CD19} P.T. Chruściel and E. Delay, {\em The hyperbolic positive energy theorem}, arXiv preprint, arXiv:1901.05263 (2019),  to appear in JEMS.  

\bibitem{CGNP18} P.T. Chruściel, G.J. Galloway, L. Nguyen, and T.-T. Paetz, {\em On the mass
aspect function and positive energy theorems for asymptotically hyperbolic
manifolds}, Class. Quantum Grav. (2018), in press, arXiv:1801.03442. 

\bibitem{CG21} P.T. Chruściel and G.J. Galloway {\em Positive mass theorems for asymptotically hyperbolic Riemannian manifolds with boundary} Class. Quantum Grav.  \textbf{38}  (2021) 237001.  

\bibitem{CH03}
P.T. Chruściel, and M. Herzlich, {\em The mass of asymptotically hyperbolic Riemannian manifolds}, Pacific journal of mathematics \textbf{212} (2003), no. 2, 231-264.

\bibitem{DGS} M. Dahl, R. Gicquaud, and A. Sakovich, {\em Asymptotically hyperbolic manifolds with small mass}, Communications in Mathematical Physics \textbf{325} (2014), no. 2, 757-801.

\bibitem{DS15} M. Dahl and A. Sakovich, {\em A density theorem for asymptotically hyperbolic initial data satisfying the dominant energy condition}, arXiv preprint, arXiv:1502.07487 (2015). 

\bibitem{E10} M. Eichmair, {\em Existence, regularity, and properties of generalized apparent horizons}, Comm. Math. Phys. \textbf{294} (2010), no. 3, 745–760. 

\bibitem{E3}
M. Eichmair, {\em The Jang equation reduction of the spacetime positive energy theorem in dimensions less than eight}, Communications in Mathematical Physics \textbf{319} (2013), no. 3, 575-593.

%\bibitem{EGM20} M. Eichmair, G.J. Galloway, and A. Mendes, {\em Initial data rigidity results}, Comm. Math. Phys. (2020), Online First; arXiv:2009.09527. 

\bibitem{EHLS}
M. Eichmair, L.-H. Huang, D. Lee, and R. Schoen, {\em The spacetime positive mass theorem in dimensions less than eight}, J. Eur. Math. Soc. (JEMS) \textbf{18} (2016), no. 1, 83–121.

\bibitem{FCS80}
D. Fischer‐Colbrie, and R. Schoen, {\em The structure of complete stable minimal surfaces in 3‐manifolds of non‐negative scalar curvature}, Communications on Pure and Applied Mathematics \textbf{33} (1980), no. 2, 199-211. 

\bibitem{GL}
G. J. Galloway, and D. Lee, {\em A note on the positive mass theorem with boundary},  Letters in Mathematical Physics \textbf{111} (2021), no. 4, 111.

%\bibitem{GS06} G. J. Galloway, and R. Schoen, {\em A generalization of Hawking’s black hole topology theorem to higher dimensions}, Comm. Math. Phys. \textbf{266} (2006), no. 2, 571–576.

\bibitem{GallowayTsang} 
G. J. Galloway, and T.-Y. Tsang, {\em Positive mass theorems for manifolds with ALH toroidal ends}, arXiv preprint, arXiv:2602.08789 (2026).  

\bibitem{G}
R. Gicquaud, {\em De l'équation de prescription de courbure scalaire aux équations de contrainte en relativité générale sur une variété asymptotiquement hyperbolique}, Journal de mathématiques pures et appliquées \textbf{94} (2010), no. 2, 200-227.

\bibitem{GT}
D. Gilbarg, and N. S. Trudinger, {\em Elliptic partial differential equations of second order}, Vol. 224, no. 2. Berlin: springer, 1977. 

\bibitem{GS}
M. Goering, and A. Skorobogatova, {\em Flat interior singularities for area almost-minimizing currents}, arXiv preprint, arXiv:2309.09634 (2023).

\bibitem{HJ} 
L.-H. Huang, and H.-C. Jang, {\em Scalar curvature deformation and mass rigidity for ALH manifolds with boundary}, Transactions of the American Mathematical Society \textbf{375} (2022), no. 11, 8151-8191.

\bibitem{HJM} L.-H. Huang, H.C. Jang, and D. Martin. {\em Mass rigidity for hyperbolic manifolds}, Communications in mathematical physics \textbf{376} (2020), no. 3, 2329-2349. 

\bibitem{HKLZ} S. Hirsch, M. Khuri, M. Lesourd, and Y. Zhang, {\em The Hyperboloidal and Spacetime Positive Mass Theorem in All Dimensions}, arXiv preprint, arXiv:2604.24746 (2026). 

\bibitem{HSY} S. He, Y. Shi, H. Yu. {\em Singularity removal rigidity theorems for minimal hypersurfaces in manifolds with nonnegative scalar curvature}, arXiv preprint, arXiv:2602.23705 (2026). 

\bibitem{Jang} P. S. Jang, {\em On the positivity of energy in general relativity}, Journal of Mathematical Physics \textbf{19} (1978), no. 5, 1152-1155.

\bibitem{LLU22} D. Lee, M. Lesourd, and R. Unger, {\em Density and positive mass theorems for incomplete manifolds}, Calculus of Variations and Partial Differential Equations \textbf{62} (2023), no. 7, 194. 

\bibitem{LLU21} D. Lee, M. Lesourd, and R. Unger, {\em Density and positive mass theorems for initial data sets with boundary}, Comm. Math. Phys., https://link.springer.com/article/10.1007/s00220-022-04439-1

\bibitem{Lee}
J. M. Lee, {\em Fredholm operators and Einstein metrics on conformally compact manifolds}, Vol. 13. American Mathematical Soc., 2006.

\bibitem{LUY21} M. Lesourd, R. Unger, and S-T. Yau,  {\em The Positive Mass Theorem with Arbitrary Ends},  to appear in Journal. Diff. Geom.; arXiv:2103.02744 (2021). 

\bibitem{L}
D. Lundberg, {\em On Jang's equation and the Positive Mass Theorem for asymptotically hyperbolic initial data sets with dimensions above three and below eight}, arXiv preprint, arXiv:2309.11330 (2023).

\bibitem{MO}
M. Min-Oo, {\em Scalar curvature rigidity of asymptotically hyperbolic spin manifolds},
Math. Ann. \textbf{285} (1989), 527–539. 

\bibitem{NV}
A. Naber, and D. Valtorta, {\em The singular structure and regularity of stationary and minimizing varifolds}, J. Eur. Math. Soc.  \textbf{22} (2020), no. 10, pp. 3305–3382. 

\bibitem{PT}
T. Parker, and C. H. Taubes, {\em On Witten's proof of the positive energy theorem}, Communications in Mathematical Physics \textbf{84} (1982), no. 2, 223-238.

\bibitem{S}
A. Sakovich, {\em The Jang equation and the positive mass theorem in the asymptotically hyperbolic setting}, Communications in Mathematical Physics \textbf{386} (2021), no. 2, 903-973.

\bibitem{SY1}
R. Schoen, and S.-T. Yau, {\em On the proof of the positive mass conjecture in general relativity}, Communications in Mathematical Physics \textbf{65} (1979), no. 1, 45-76.

\bibitem{SY6}
R. Schoen, Richard, and S.-T. Yau, {\em The energy and the linear momentum of space-times in general relativity}, Communications in Mathematical Physics \textbf{79} (1981), no. 1, 47-51.

\bibitem{SY2}
R. Schoen, and S.-T. Yau, {\em Proof of the positive mass theorem. II}, Communications in Mathematical Physics \textbf{79} (1981), no .2, 231-260.

\bibitem{SY4}
R. Schoen, and S.-T. Yau, {\em Conformally flat manifolds, Kleinian groups and scalar curvature}, Inventiones mathematicae \textbf{92} (1988), no. 1, 47-71. 

\bibitem{SY5}
R. Schoen, and S.-T. Yau, {\em Lectures on Differential Geometry}, International Press (1994). 

\bibitem{SY7}
R. Schoen, and S.-T. Yau, {\em Positive scalar curvature and minimal hypersurface singularities}, Surv. Differ. Geom. vol. 24, 441–480, International Press, Boston, MA,
(2022).

\bibitem{XW}
X. Wang, {\em The mass of asymptotically hyperbolic manifolds}, Journal of Differential Geometry \textbf{57} (2001), no. 2, 273-299. 

\bibitem{W}
E. Witten, {\em A new proof of the positive energy theorem}, Communications in Mathematical Physics \textbf{80} (1981), no. 3, 381-402. 

\bibitem{Zhu}
J. Zhu, {\em Positive mass theorem with arbitrary ends and its application},  International Mathematics Research Notices 2023, no. 11 (2023), 9880-9900.

\end{thebibliography}
\end{document}